\documentclass[portuges,12pt,letter]{article}
\usepackage[centertags]{amsmath}
\usepackage{amsfonts}
\usepackage{newlfont}
\usepackage{amscd}
\usepackage{graphics}
\usepackage{epsfig}
\usepackage{indentfirst}
\usepackage{amsxtra}
\usepackage[latin1]{inputenc}
\usepackage{amssymb, amsmath}
\usepackage{amsthm}
\usepackage[mathscr]{eucal}

\newtheorem{thm}{Theorem}[section]

\newtheorem{remark}[thm]{Remark}

\author{Fabio Silva Botelho \\ Department of Mathematics \\ Federal University of  Santa Catarina, UFSC \\
Florian\'{o}polis, SC - Brazil}

\title{\bf A primal dual variational formulation and a multi-duality principle for a  non-linear model of plates} 
\date{}
\begin{document}
\maketitle

\abstract{
This article develops a new primal dual formulation for the Kirchhoff-Love non-linear plate model.

At first we establish a duality principle which includes sufficient conditions of global optimality through the
dual formulation. At this point we highlight this first duality principle is specially suitable for the case in which the membrane stress tensor is negative definite.   In a second step, from such a general principle, we develop a primal dual variational formulation which also includes
the corresponding sufficient conditions for global optimality. The results are based on standard tools of convex analysis  and on a well known Toland result for D.C.
optimization.

Finally, in the last section, we present a multi-duality principle and qualitative relations between the critical points of the primal and dual
formulations.
We formally prove there is no duality gap between such primal and dual formulations in a local extremal context.
}

\section{Introduction} In this work we develop a new primal dual variational formulation for the Kirchhoff-Love non-linear plate model.
We emphasize the results here presented may be applied to a large class of non-convex variational problems.

At this point we start to describe the primal formulation.

    Let $\Omega\subset\mathbb{R}^2$ be an open, bounded, connected set  which  represents the middle surface of a plate
of thickness $h$. The boundary of $\Omega$, which is assumed to be regular (Lipschitzian), is
denoted by $\partial \Omega$. The vectorial basis related to the cartesian
system $\{x_1,x_2,x_3\}$ is denoted by $( \textbf{a}_\alpha,
\textbf{a}_3)$, where $\alpha =1,2$ (in general Greek indices stand
for 1 or 2), and where $\textbf{a}_3$ is the vector normal to $\Omega$, whereas $\textbf{a}_1$ and $\textbf{a}_2$ are orthogonal vectors parallel to $\Omega.$  Also,
$\textbf{n}$ is
the outward normal to the plate surface.

    The displacements will be denoted by
$$ \hat{\textbf{u}}=\{\hat{u}_\alpha,\hat{u}_3\}=\hat{u}_\alpha
\textbf{a}_\alpha+ \hat{u}_3 \textbf{a}_3.
$$

The Kirchhoff-Love relations are
\begin{eqnarray}&&\hat{u}_\alpha(x_1,x_2,x_3)=u_\alpha(x_1,x_2)-x_3
w(x_1,x_2)_{,\alpha}\; \nonumber \\ && \text{ and }
\hat{u}_3(x_1,x_2,x_3)=w(x_1,x_2).
\end{eqnarray}

Here $ -h/2\leq x_3 \leq h/2$ so that we have
$u=(u_\alpha,w)\in U$ where
\begin{eqnarray}
U&=&\left\{(u_\alpha,w)\in W^{1,2}(\Omega; \mathbb{R}^2) \times
W^{2,2}(\Omega) , \right.\nonumber \\ &&\; u_\alpha=w=\frac{\partial w}{\partial \textbf{n}}=0 \left.
\text{ on } \partial \Omega \right\} \nonumber \\ &=&W_0^{1,2}(\Omega; \mathbb{R}^2) \times W_0^{2,2}(\Omega).\nonumber\end{eqnarray}
It is worth emphasizing that the boundary conditions here specified refer to a clamped plate.

We define the operator $\Lambda: U \rightarrow Y_1 \times Y_1$, where $Y_1=Y^*_1=L^2(\Omega; \mathbb{R}^{2 \times 2})$, by
$$\Lambda(u)=\{\gamma(u), \kappa(u)\},$$
$$\gamma_{\alpha\beta}(u)= \frac{u_{\alpha,\beta}+u_{\beta,\alpha}}{2}+\frac{w_{,\alpha} w_{,\beta}}{2},$$
$$\kappa_{\alpha \beta}(u)=-w_{,\alpha \beta}.$$
The constitutive relations are given by
\begin{equation}
N_{\alpha\beta}(u)=H_{\alpha\beta\lambda\mu} \gamma_{\lambda\mu}(u),
\end{equation}
\begin{equation}
M_{\alpha\beta}(u)=h_{\alpha\beta\lambda\mu} \kappa_{\lambda\mu}(u),
\end{equation}
where: $\{H_{\alpha\beta\lambda\mu}\}$
 and
 $\{h_{\alpha\beta\lambda\mu}=\frac{h^2}{12}H_{\alpha\beta\lambda\mu}\}$,
 are symmetric positive definite   fourth order tensors. From now on, we denote $\{\overline{H}_{\alpha\beta \lambda \mu}\}=\{H_{\alpha\beta \lambda \mu}\}^{-1}$ and $\{\overline{h}_{\alpha\beta \lambda \mu}\}=\{h_{\alpha\beta \lambda \mu}\}^{-1}$.

 Furthermore
 $\{N_{\alpha\beta}\}$ denote the membrane stress tensor and
 $\{M_{\alpha\beta}\}$ the moment one.
    The plate stored energy, represented by $(G\circ
    \Lambda):U\rightarrow\mathbb{R}$ is expressed by
 \begin{equation}\label{80} (G\circ\Lambda)(u)=\frac{1}{2}\int_{\Omega}
  N_{\alpha\beta}(u)\gamma_{\alpha\beta}(u)\;dx+\frac{1}{2}\int_{\Omega}
  M_{\alpha\beta}(u)\kappa_{\alpha\beta}(u)\;dx
  \end{equation}
 and the external work, represented by $F_1:U\rightarrow\mathbb{R}$, is given by
   \begin{equation}\label{81} F_1(u)=\langle w,P \rangle_{L^2(\Omega)}+\langle u_\alpha,P_\alpha \rangle_{L^2(\Omega)}
,\end{equation} where $P, P_1, P_2 \in L^2(\Omega)$ are external loads in the directions $\textbf{a}_3$, $\textbf{a}_1$ and $\textbf{a}_2$ respectively. The potential energy, denoted by
$J:U\rightarrow\mathbb{R}$ is expressed by:
$$
J(u)=(G\circ\Lambda)(u)-F_1(u)
$$

Finally, we also emphasize from now on, as their meaning are clear, we may denote $L^2(\Omega)$ and $L^2(\Omega; \mathbb{R}^{2 \times 2})$ simply by $L^2$, and the respective norms by $\|\cdot \|_2.$ Moreover, unless otherwise indicated, derivatives are always understood in the distributional sense, $\mathbf{0}$ may denote the zero vector in  appropriate Banach spaces and, the following and relating notations are used:
$$w_{,\alpha\beta}=\frac{\partial^2 w}{\partial x_\alpha \partial x_\beta},$$
$$u_{\alpha,\beta}=\frac{\partial u_\alpha}{\partial x_\beta},$$
$$N_{\alpha\beta,1}=\frac{\partial N_{\alpha\beta}}{\partial x_1},$$
and
$$N_{\alpha\beta,2}=\frac{\partial N_{\alpha\beta}}{\partial x_2}.$$

Here we emphasize the general Einstein convention of sum of repeated indices holds throughout  the text, unless otherwise indicated.
\begin{remark} About the references, details on the Sobolev involved may be found in \cite{1}. Mandatory references  are the original results of
Telega and his co-workers in
\cite{2900, 85,10,11}. About convex analysis, the results here developed follow in some extent \cite{120}, for which the main references are \cite{[6],12}.

We emphasize, our results  complement, in some sense, the original ones presented in \cite{2900,85,10,11}.

Finally, existence results for models in elasticity including the plate model here addressed are developed in \cite{903,[3],[4]}.
Similar problems are addressed in \cite{77,9}.
\end{remark}
\section{The first duality principle}
\begin{thm} Let $\Omega \subset \mathbb{R}^2$ be an open, bounded, connected set with a regular (Lipschitzian) boundary denoted by
$\partial \Omega$. Let $U=U_1 \times U_2$ where $U_1=W_0^{1,2}(\Omega;\mathbb{R}^2)$ and $U_2=W_0^{2,2}(\Omega)$. We recall that
$Y_1=Y_1^*=L^2(\Omega; \mathbb{R}^{2 \times 2})$ and define $Y_2=Y_2^*=L^2(\Omega;\mathbb{R}^2)$,
$$\gamma_{\alpha\beta}(u)=\frac{1}{2}(u_{\alpha,\beta}+u_{\beta,\alpha})+\frac{1}{2} w_{,\alpha}w_{,\beta},$$

$$\kappa_{\alpha\beta}(u)=-w_{,\alpha\beta}, \; \forall u \in U,\;\;\alpha, \beta \in \{1,2\}$$
and $J:U \rightarrow \mathbb{R}$ by,
$$J(u)=G_1(\gamma (u))+G_2(\kappa (u))-\langle u, f \rangle_{L^2},$$ where

$$G_1(\gamma (u))=\frac{1}{2}\int_\Omega H_{\alpha\beta\lambda\mu} \gamma_{\alpha\beta}(u) \gamma_{\lambda \mu}(u)\;dx,$$
$$G_2(\kappa (u))=\frac{1}{2}\int_\Omega h_{\alpha\beta\lambda\mu} \kappa_{\alpha\beta}(u) \kappa_{\lambda \mu}(u)\;dx,$$
and
$$\langle u,f \rangle_{L^2}=\langle w,P \rangle_{L^2}+\langle u_\alpha, P_\alpha \rangle_{L^2},$$
where
$$f=(P_1,P_2,P) \in L^2(\Omega ; \mathbb{R}^3).$$

In the next lines we shall denote $$w=-(h_{\alpha\beta\lambda\mu} D_{\lambda \mu}^*D_{\alpha\beta})^{-1}( \text{ div }Q+ \text{ div }z^*-P),$$ if
$$\text{ div }Q+ \text{ div }z^*-P=-h_{\alpha\beta\lambda\mu}w_{,\alpha\beta\lambda\mu}$$
and $$ w \in U_2.$$

Define also $G_2^*:Y_2^* \times Y_2^* \rightarrow \mathbb{R},$ by

\begin{eqnarray}&&G_2^*(z^*,Q)\nonumber \\ &=&\sup_{w \in U}\{\langle w_{,\alpha}, z_{\alpha}^* +Q_{\alpha}\rangle_{L^2}-
G_2(\kappa(u))+\langle P,w\rangle_{L^2}\}\nonumber  \\ &=& \frac{1}{2}\int_\Omega [(h_{\alpha\beta\lambda\mu} D_{\lambda \mu}^*D_{\alpha\beta})^{-1}( \text{ div }Q+ \text{ div }z^*-P)]
[( \text{ div } Q + \text{ div } z^*-P)]\;dx, \nonumber\end{eqnarray}
and $\tilde{G}_1^*: Y_2^* \times Y_1^* \rightarrow \mathbb{R} \cup \{+\infty\}$ where
\begin{eqnarray}\tilde{G}_1^*(-Q,N)&=&\sup_{ (v_1,v_2) \in Y_1 \times Y_2}\left\{- \langle (v_2)_\alpha, Q_\alpha \rangle_{L^2}+\langle (v_1)_{\alpha\beta},N_{\alpha\beta}\rangle_{L^2}
\right. \nonumber \\ && \left.-G_1\left(\left\{(v_1)_{\alpha\beta}+\frac{1}{2}(v_2)_\alpha (v_2)_\beta\right\}\right)-\frac{1}{2}\int_\Omega K_{\alpha\beta}(v_2)_\alpha (v_2)_\beta\;dx \right\} \nonumber \\ &=& \frac{1}{2}\int_\Omega \overline{N_{\alpha\beta}+K_{\alpha \beta}}Q_\alpha Q_\beta\;dx
\nonumber \\ && +\frac{1}{2}\int_\Omega \overline{H}_{\alpha\beta\lambda\mu}N_{\alpha \beta}N_{\lambda \mu}\;dx
\end{eqnarray}
if $$\{N_{\alpha\beta}+K_{\alpha\beta}\}$$ is positive definite.

Here we have denoted,
$$\{\overline{N_{\alpha\beta}+K_{\alpha\beta}}\}=\{{N_{\alpha\beta}+K_{\alpha\beta}}\}^{-1}.$$

We denote also,
\begin{eqnarray}
F^*(z^*)&=& \sup_{v_2 \in Y_2} \left\{ \langle (v_2)_\alpha, (z_2^*)_{\alpha} \rangle_{L^2}-\frac{1}{2}\int_\Omega K_{\alpha\beta}(v_2)_\alpha (v_2)_\beta\;dx
\right.\nonumber \\ &=& \left.\frac{1}{2}\int_\Omega \overline{K}_{\alpha\beta} (z_2^*)_{\alpha} (z_2^*)_{\beta}\;dx \right.
\nonumber \\ &=& \left.\frac{1}{2}\int_\Omega \overline{-N_{\alpha\beta}+\varepsilon \delta_{\alpha\beta}} (z_2^*)_{\alpha} (z_2^*)_{\beta}\;dx\right\},
\end{eqnarray}
if $\{K_{\alpha\beta}\}$ is positive definite, where
$$\{\overline{K}_{\alpha\beta}\}=\{K_{\alpha\beta}\}^{-1},$$
$$\{K_{\alpha\beta}\}=\{-N_{\alpha_\beta}+\varepsilon \delta_{\alpha\beta}\},$$
$$\{\overline{H}_{\alpha\beta\lambda\mu}\}=\{H_{\alpha\beta\lambda\mu}\}^{-1},$$
and $$F(\{w_\alpha\})=\frac{1}{2} \int_\Omega K_{\alpha\beta} w_{,\alpha}w_{,\beta}\;dx,$$
for some $$\varepsilon>0.$$

At this point we also define,
$$B^*=\{ N \in Y_1^*\;:\;\{K_{\alpha\beta}\} \text{ is positive definite and } \hat{J}^*(N,z^*)>0, \forall z^* \in Y_2^* \text{ such that } z^* \neq \mathbf{0}\}$$
where,
\begin{eqnarray}\hat{J}^*(N,z^*)&=& F^*(z^*)-G^*_2(z^*,\mathbf{0}) \nonumber \\ &=& \frac{1}{2}\int_\Omega \overline{(-N_{\alpha\beta}+ \varepsilon \delta_{\alpha\beta})}
z_\alpha^*z_\beta^*\;dx \nonumber \\ &&- \frac{1}{2}\int_\Omega [(h_{\alpha\beta\lambda\mu} D_{\lambda \mu}^*D_{\alpha\beta})^{-1}( \text{ div } z^*)][ \text{ div }z^*]\;dx.
\end{eqnarray}
Moreover, we denote,
$$C^*=\{N \in Y_1^*\;:\; N_{\alpha\beta,\beta}+P_\alpha=0, \text{ in } \Omega\}$$
and define $$A^*=B^* \cap C^*.$$
Assume $u_0 \in U$ is such that $\delta J(u_0)=\mathbf{0}$ and $N_0 \in B^*$, where
$$\{(N_0)_{\alpha\beta}\}=\{H_{\alpha\beta\lambda\mu}\gamma_{\lambda \mu}(u_0)\}.$$

Under such hypotheses,

\begin{eqnarray}
J(u_0)&=& \inf_{u \in U} J(u) \nonumber \\ &=& \sup_{(Q,N) \in Y_2^*\times A^*} \tilde{J}^*(Q,N) \nonumber \\ &=&
\tilde{J}^*(Q_0,N_0) \nonumber \\ &=& J^*(z_0^*,Q_0,N_0),
\end{eqnarray}
where
$$J^*(z^*,Q,N)=F^*(z^*)-G_2^*(z^*,Q)-\tilde{G}_1^*(-Q,N),$$
$$\tilde{J}^*(Q,N)=\inf_{ z^* \in Y_2^*} J^*(z^*,Q,N),$$
and where
$$(z_0^*)_{\alpha}=(-(N_0)_{\alpha\beta}+ \varepsilon \delta_{\alpha\beta})(w_0)_{,\beta},$$
and
$$Q_0=-\varepsilon \nabla w_0.$$
\end{thm}
\begin{proof}
From the general result in Toland \cite{12}, we have
\begin{eqnarray}
&&\inf_{z^* \in Y_2^*} J^*(z^*,Q,N)
\nonumber \\ &=& \inf_{z^* \in Y_2^*}\{F^*(z^*)-G_2^*(z^*,Q)-\tilde{G}_1^*(-Q,N)\}
\nonumber \\ &\leq& -\langle w_{,\alpha},z^*_\alpha \rangle_{L^2}- \langle w_{,\alpha}, Q_{\alpha}\rangle_{L^2}
+\frac{1}{2}\int_\Omega h_{\alpha\beta\lambda\mu} w_{,\alpha\beta}w_{,\lambda\mu}\;dx
\nonumber \\ && +\langle w_{,\alpha},z^*_\alpha \rangle_{L^2}-\frac{1}{2}\int_\Omega(-N_{\alpha\beta}+\varepsilon \delta_{\alpha\beta})w_{,\alpha}w_{,\beta}\;dx
-\langle w,P \rangle_{L^2}\nonumber \\ && - \frac{1}{2\varepsilon} \int_\Omega \delta_{\alpha\beta}Q_{\alpha}
 Q_\beta\;dx \nonumber \\ && -\frac{1}{2}\int_\Omega \overline{H}_{\alpha\beta\lambda\mu}N_{\alpha\beta}N_{\lambda \mu}\;dx
 -\langle u_\alpha, N_{\alpha \beta,\beta}+P_\alpha \rangle_{L^2}
 \nonumber \\ &\leq& -\langle w_{,\alpha},z^*_\alpha \rangle_{L^2}- \langle w_{,\alpha}, Q_{\alpha}\rangle_{L^2}
+\frac{1}{2}\int_\Omega h_{\alpha\beta\lambda\mu} w_{,\alpha\beta}w_{,\lambda\mu}\;dx
\nonumber \\ && +\langle w_{,\alpha},z^*_\alpha \rangle_{L^2}-\frac{1}{2}\int_\Omega(-N_{\alpha\beta}+\varepsilon \delta_{\alpha\beta})w_{,\alpha}w_{,\beta}\;dx
-\langle w,P \rangle_{L^2}\nonumber \\ && +\langle w_{,\alpha},Q_\alpha^* \rangle_{L^2}+\frac{1}{2} \int_\Omega \varepsilon \delta_{\alpha\beta}w_{,\alpha}
 w_{,\beta}\;dx \nonumber \\ && -\frac{1}{2}\int_\Omega \overline{H}_{\alpha\beta\lambda\mu}N_{\alpha\beta}N_{\lambda \mu}\;dx
 -\langle u_\alpha, N_{\alpha_\beta,\beta}+P_\alpha \rangle_{L^2},
 \end{eqnarray}
 $\forall u \in U,\; Q \in Y_2^*,\; N \in A^*.$

 Therefore,
 \begin{eqnarray}
 &&\inf_{z^* \in Y_2^*} J^*(z^*,Q,N)
 \nonumber \\ &\leq& \frac{1}{2}\int_\Omega h_{\alpha\beta\lambda\mu} w_{,\alpha\beta}w_{,\lambda\mu}\;dx
+\frac{1}{2}\int_\Omega N_{\alpha\beta} w_{,\alpha}w_{,\beta}\;dx
\nonumber \\ &&  +\left\langle \left(\frac{u_{\alpha,\beta}+u_{\beta,\alpha}}{2}\right), N_{\alpha\beta} \right\rangle_{L^2}
-\frac{1}{2}\int_\Omega \overline{H}_{\alpha\beta\lambda\mu} N_{\alpha\beta}N_{\lambda\mu}\;dx \nonumber \\ &&
-\langle w,P\rangle_{L^2}-\langle u_\alpha,P_\alpha \rangle_{L^2}
\nonumber \\ &\leq&
\sup_{N \in Y^*_1}\left\{ \left\langle \frac{u_{\alpha,\beta}+u_{\beta,\alpha}}{2}+\frac{1}{2}w_{,\alpha} w_{,\beta}, N_{\alpha\beta} \right\rangle_{L^2}
-\frac{1}{2}\int_\Omega \overline{H}_{\alpha\beta\lambda\mu} N_{\alpha\beta}N_{\lambda\mu}\;dx \right\}
\nonumber \\ &&+\frac{1}{2}\int_\Omega h_{\alpha\beta\lambda\mu} w_{,\alpha\beta}w_{,\lambda\mu}\;dx \nonumber \\ &&
-\langle w,P\rangle_{L^2}-\langle u_\alpha,P_\alpha \rangle_{L^2} \nonumber \\ &=&
\frac{1}{2} \int_\Omega \gamma_{\alpha\beta}(u)\gamma_{\lambda \mu}(u)\;dx+\frac{1}{2}\int_\Omega h_{\alpha\beta\lambda\mu} \kappa_{\alpha\beta}(u)\kappa_{\lambda\mu}(u)\;dx \nonumber \\ &&
-\langle w,P\rangle_{L^2}-\langle u_\alpha,P_\alpha \rangle_{L^2} \nonumber \\ &=& J(u), \end{eqnarray}
$\forall u \in U,\; Q \in Y_2^*,\; N \in A^*.$

Summarizing,
\begin{eqnarray}\label{q2}
J(u) &\geq& \inf_{z^* \in Y_2^*} J^*(z^*,Q,N)
\nonumber \\ &=& \tilde{J}^*(Q,N) \end{eqnarray}
$\forall u \in U,\; Q \in Y_2^*, \; N \in A^*,$ so that

\begin{equation}\label{q3}
\inf_{u \in U} J(u) \geq \sup_{(Q,N) \in Y_2^* \cap A^*} \tilde{J}^*(Q,N). \end{equation}

Suppose now $u_0 \in U$ is such that
\begin{equation}\label{q4} \delta J(u_0)=\mathbf{0},\end{equation}
and
$N_0 \in B^*.$

Observe that, from (\ref{q4}),
$$(N_0)_{\alpha\beta,\beta}+P_\alpha=0, \text{ in } \Omega,$$
so that
$$N_0 \in C^*.$$

Hence, $$N_0 \in A^*=B^* \cap C^*.$$

Moreover, from  $\delta J(u_0)=\mathbf{0}$, we have
\begin{equation}\label{q5}-\text{ div } (Q_0+z_0^*)=h_{\alpha\beta\lambda\mu}(w_0)_{\alpha\beta\lambda\mu}-P \text{ in } \Omega,\end{equation}
where, as above indicated
\begin{equation}\label{q6}(z_0^*)_\alpha=(-(N_0)_{\alpha\beta}+\varepsilon \delta_{\alpha\beta})(w_0)_{,\beta},\end{equation}
and $$Q_0=-\varepsilon \nabla w_0.$$
From (\ref{q5}), $$w_0=- (h_{\alpha\beta\lambda\mu}D_{\lambda\mu}^*D_{\alpha\beta})^{-1}(\text{ div}(z_0^*+Q_0)-P),$$
so that from this and the inversion of (\ref{q6}), we have
\begin{eqnarray}
(w_0)_{,\rho}&=&[(h_{\alpha\beta\lambda\mu}D_{\lambda\mu}^*D_{\alpha\beta})^{-1}(-\text{ div}(z_0^*+Q_0)+P)]_{,\rho}
\nonumber \\ &=& \overline{((-N_0)_{\rho\beta}+\varepsilon \delta_{\rho\beta})} (z_0^*)_\beta,
\end{eqnarray}
so that $$[(h_{\alpha\beta\lambda\mu}D_{\lambda\mu}^*D_{\alpha\beta})^{-1}(\text{ div}(z_0^*+Q_0)-P)]_{,\rho}
+\overline{((-N_0)_{\rho\beta}+\varepsilon \delta_{\rho\beta})} (z_0^*)_\beta=0, \text{ in } \Omega,$$ that is,
$$\frac{\partial \hat{J}_1^*(z_0^*,Q_0,w_0,u_0)}{\partial z^*}=\mathbf{0},$$
where
$$\hat{J}^*_1(z^*,Q,N,u)=J^*(z^*,Q,N)+\langle u_\alpha,  N_{\alpha\beta,\beta}+P_\alpha \rangle_{L^2}.$$

Also, from (\ref{q5}) and (\ref{q6}) we obtain
 \begin{eqnarray}
 -\text{ div }Q_0 &=& \varepsilon \text{ div }\nabla w_0
 \nonumber \\ &=& \varepsilon \nabla^2 w_0 \nonumber \\ &=&\varepsilon \delta_{\alpha \beta} w_{,\alpha\beta}
  \nonumber \\ &=&h_{\alpha\beta\lambda\mu} (w_0)_{\alpha\beta\lambda\mu}-[((N_0)_{\alpha\beta}-\varepsilon \delta_{\alpha\beta})(w_0)_{,\beta}]_{,\alpha}-P
 \nonumber \\ &=&  h_{\alpha\beta\lambda\mu} (w_0)_{\alpha\beta\lambda\mu}+\text{ div }z_0^* -P.
\end{eqnarray}
Furthermore,
\begin{eqnarray}\frac{(Q_0)_\rho}{\varepsilon}&=&-(w_0)_{,\rho}
\nonumber \\ &=&[(h_{\alpha\beta\lambda\mu}D_{\lambda\mu}^*D_{\alpha\beta})^{-1}(\text{ div}(z_0^*+Q_0)-P)]_{,\rho}.
\end{eqnarray}

This last equation corresponds to
$$\frac{\partial \hat{J}^*_1(z_0^*,Q_0,N_0,u_0)}{\partial Q_\alpha}=0.$$

Moreover
$$\overline{H}_{\alpha\beta\lambda\mu} (N_0)_{\lambda\mu}=\left(\frac{(u_0)_{\alpha,\beta}+(u_0)_{\beta,\alpha}}{2}\right)+\frac{1}{2} (w_0)_{,\alpha} (w_0)_{,\beta},$$
so that
\begin{eqnarray}
&&\overline{H}_{\alpha\beta\lambda\mu} (N_0)_{\lambda\mu}\nonumber \\ &=&\left(\frac{(u_0)_{\alpha,\beta}+(u_0)_{\beta,\alpha}}{2}\right)
\nonumber \\ &&+ \left(\frac{1}{2}\overline{(-(N_0)_{\alpha \rho}+\varepsilon \delta_{\alpha\rho})} (z_0^*)_\rho\right)\left(
 \overline{(-(N_0)_{\beta \eta}+\varepsilon \delta_{\beta\eta})} (z_0^*)_\eta\right),
 \end{eqnarray}
 which means
 $$\frac{\partial \hat{J}_1^*(z_0^*,Q_0,N_0,u_0)}{\partial N_{\alpha\beta}}=0.$$

 Finally, from $$N_{\alpha\beta, \beta}+P_\alpha=0, \text{ in } \Omega,$$ we get
 $$\frac{\partial \hat{J}^*_1(z_0^*,Q_0,N_0,u_0)}{\partial u_\alpha}=0.$$

 Summarizing, we have obtained
 $$\delta \hat{J}^*_1(z_0^*,Q_0,N_0,u_0)=\mathbf{0}.$$

 At this point we shall obtain a standard correspondence between the primal and dual formulations.

 First, we recall that from $$(z_0^*)_\alpha= (-N_{\alpha\beta}+\varepsilon \delta_{\alpha\beta})(w_0)_{,\beta},$$
 we have
 $$F^*(z^*_0)=\langle (w_0)_\alpha,(z_0^*)_\alpha \rangle_{L^2}-\frac{1}{2}(-N_{\alpha\beta}+\varepsilon \delta_{\alpha\beta})(w_0)_{,\alpha}(w_0)_{,\beta}.$$

 From $$\text{ div } Q_0+\text{ div } z_0^*=-h_{\alpha\beta \lambda \mu}(w_0)_{\alpha\beta\lambda \mu}+P$$
 we obtain
 $$G_2^*(z_0^*,Q_0)= \langle (w_0)_\alpha,(Q_0)_\alpha \rangle_{L^2} +\langle (w_0)_\alpha,(z_0^*)_\alpha \rangle_{L^2}-G_2(\kappa(u_0))+\langle w_0,P\rangle_{L^2}.$$

 Finally, from $$\frac{-(Q_0)_\alpha}{\varepsilon}=(w_0)_{,\alpha}$$ and $$(N_0)_{\alpha\beta}=H_{\alpha\beta\lambda\mu} \gamma_{\lambda\mu}(u_0),$$
 we get
 \begin{eqnarray} \tilde{G}_1^*(-Q_0,N_0)&=& -\langle (w_0)_\alpha,(Q_0)_\alpha \rangle_{L^2}+\left\langle \frac{(u_0)_{\alpha,\beta}+(u_0)_{\beta,\alpha}}{2}+\frac{(w_0)_{,\alpha}(w_0)_{,\beta}}{2},(N_0)_{\alpha\beta} \right\rangle_{L^2} \nonumber \\ &&-G_1(\gamma(u_0))-\frac{\varepsilon}{2}\int_\Omega \delta_{\alpha\beta}(w_0)_{,\alpha}(w_0)_{,\beta}\;dx \nonumber \\ &=&-\langle (w_0)_\alpha,(Q_0)_\alpha \rangle_{L^2}+\left\langle\frac{(w_0)_{,\alpha}(w_0)_{,\beta}}{2},(N_0)_{\alpha\beta}\right\rangle_{L^2} \nonumber \\ &&-\frac{\varepsilon}{2}\int_\Omega \delta_{\alpha\beta}(w_0)_{,\alpha}(w_0)_{,\beta}\;dx+\langle u_\alpha,P_\alpha\rangle_{L^2}
 \nonumber \\ && -\frac{1}{2}\int_\Omega H_{\alpha\beta\lambda\mu} \gamma_{\alpha\beta}(u_0)\gamma_{\lambda \mu}(u_0)\;dx.
 \end{eqnarray}

 Joining the pieces, we obtain
 \begin{eqnarray}\label{q10}
 \hat{J}^*_1(z_0^*,Q_0,N_0,u_0) &=& J^*(z_0,Q_0,N_0) \nonumber \\ &=& F^*(z_0^*)-G_2^*(z_0^*,Q_0)-\tilde{G}_1^*(-Q_0,N_0)
 \nonumber \\ &=& G_2(\kappa u_0)+G_1(\gamma (u_0))-\langle w_0,P \rangle_{L^2}-\langle (u_0)_\alpha, P_\alpha \rangle_{L^2}
 \nonumber \\ &=& J(u_0).\end{eqnarray}

 Moreover, since $N_0 \in A^*$, we have

 \begin{eqnarray}J^*(z_0^*,Q_0,N_0)&=& \inf_{z^* \in Y_2^*}J^*(z^*,Q_0,N_0) \nonumber \\ &=&\tilde{J}^*(Q_0,N_0).\end{eqnarray}
 From this,  (\ref{q3}) and (\ref{q10}), we obtain
  \begin{eqnarray}
J(u_0)&=& \inf_{u \in U} J(u) \nonumber \\ &=& \sup_{(Q,N) \in Y_2^*\times A^*} \tilde{J}^*(Q,N) \nonumber \\ &=&
\tilde{J}^*(Q_0,N_0) \nonumber \\ &=& J^*(z_0^*,Q_0,N_0),
\end{eqnarray}

The proof is complete.
\end{proof}
\section{The primal dual formulation and related duality principle}
At this point we present the main result of this article, which is summarized by the next theorem.
\begin{thm}Consider the notation and context of the last theorem.
Assume those hypotheses, more specifically suppose $\delta J(u_0)=\mathbf{0}$ and $N_0 \in B^*$, where
 $$\{(N_0)_{\alpha\beta}\}=\{H_{\alpha\beta\lambda\mu}\gamma_{\lambda \mu}(u_0)\},$$
$$(z_0^*)_{\alpha}=(-(N_0)_{\alpha\beta}+ \varepsilon \delta_{\alpha\beta})(w_0)_{,\beta},$$
and
$$ Q_0=-\varepsilon \nabla w_0.$$
Recall also that
$$B^*=\{ N \in Y_1^*\;:\;\{K_{\alpha\beta}\} \text{ is positive definite and } \hat{J}^*(N,z^*)>0, \forall z^* \in Y_2^* \text{ such that } z^* \neq \mathbf{0}\}$$
where,
\begin{eqnarray}\hat{J}^*(N,z^*)&=& F^*(z^*)-G^*_2(z^*,\mathbf{0}) \nonumber \\ &=& \frac{1}{2}\int_\Omega \overline{(-N_{\alpha\beta}+ \varepsilon \delta_{\alpha\beta})}
z_\alpha^*z_\beta^*\;dx \nonumber \\ &&- \frac{1}{2}\int_\Omega [(h_{\alpha\beta\lambda\mu} D_{\lambda \mu}^*D_{\alpha\beta})^{-1}( \text{ div } z^*)][ \text{ div }z^*]\;dx.
\end{eqnarray}
Moreover,
$$C^*=\{N \in Y_1^*\;:\; N_{\alpha\beta,\beta}+P_\alpha=0, \text{ in } \Omega\}$$
and $$A^*=B^* \cap C^*.$$

Under such assumptions and notation, denoting also $$\hat{Y}_2^*=\{ Q \in Y_2^*\;:\; Q=\nabla v, \text{ for some } v \in W^{2,2}_0(\Omega)\},$$
we have
\begin{eqnarray}
J(u_0)&=&\inf_{u \in U} J(u) \nonumber \\ &=& \sup_{(Q,N) \in \hat{Y}_2^* \times A^*}\tilde{J}^*(Q,N)
\nonumber \\ &=& \tilde{J}^*(Q_0,N_0) \nonumber \\ &=& J^*(z_0^*,Q_0,N_0) \nonumber \\ &=& J_3(w_0,N_0)
\nonumber \\ &=& \sup_{(w, N) \in U \times A^*} J_3(w,N)
\end{eqnarray}
where,
\begin{eqnarray}J_3(w,N)&=& -\frac{1}{2}\int_\Omega h_{\alpha\beta\lambda\mu} w_{,\alpha\beta} w_{,\lambda \mu}\;dx
\nonumber \\ && -\frac{1}{2}\int_\Omega (N_{\alpha\beta}-\varepsilon \delta_{\alpha\beta})w_{\,\alpha}w_{,\beta}\;dx
\nonumber \\ &&-\frac{1}{2 \varepsilon} \int_\Omega\left[(-\nabla^2)^{-1}\left(h_{\alpha\beta\lambda\mu} w_{,\alpha\beta\lambda\mu}
-[(N_{\alpha\beta}-\varepsilon \delta_{\alpha\beta})w_{,\beta}]_{,\alpha}-P\right)\right] \nonumber \\ && \times \left(h_{\alpha\beta\lambda\mu} w_{,\alpha\beta\lambda\mu}
-[(N_{\alpha\beta}-\varepsilon \delta_{\alpha\beta})w_{,\beta}]_{,\alpha}-P\right) \;dx\nonumber \\ &&
- \frac{1}{2}\int_\Omega \overline{H}_{\alpha\beta\lambda\mu}N_{\alpha\beta}N_{\lambda \mu}\;dx,
\end{eqnarray}
where generically we have denoted $$\hat{w}=(\nabla^{2})^{-1} \eta \text{ for } \eta \in L^2(\Omega),$$ if $\eta={\nabla}^2 \hat{w}$ and $\hat{w} \in W^{2,2}_0(\Omega).$
\end{thm}
\begin{proof}
Observe that
$$\tilde{J}^*(Q,N)=\inf_{z^* \in Y_2^*}\{F^*(z^*)-G_2^*(z^*,Q)-\tilde{G}_1^*(-Q,N)\},$$
$\forall Q \in \hat{Y}_2^*, \; N \in A^*.$

Also, such an infimum is attained through the equation
$$\frac{\partial F^*(z^*)}{\partial z^*}=\frac{\partial G^*_2(z^*,Q)}{\partial z^*},$$
that is
\begin{eqnarray}&&-\nabla[(h_{\alpha\beta\lambda\mu}D_{\lambda\mu}^*D_{\alpha\beta})^{-1}(\text{ div }(z^*+Q)-P)]\nonumber \\ &=&\{\overline{(-N_{\alpha\beta}+\varepsilon\delta_{\alpha\beta})}(z_2^*)_\beta\}, \end{eqnarray}
that is
\begin{eqnarray}&&-\text{ div }\left(\nabla(h_{\alpha\beta\lambda\mu}D_{\lambda\mu}^*D_{\alpha\beta})^{-1}(\text{ div }(z^*+Q)-P)\right)\nonumber \\ &=&[\overline{(-N_{\alpha\beta}+\varepsilon\delta_{\alpha\beta})}(z_2^*)_\beta]_{,\alpha}
\nonumber \\ &=& \nabla^2 w, \end{eqnarray}
where $$w=-(h_{\alpha\beta\lambda\mu}D_{\lambda\mu}^*D_{\alpha\beta})^{-1}(\text{ div }(z^*+Q)-P).$$

Hence,
\begin{eqnarray}
F^*(z^*)&=& \langle w_{,\alpha},z^*_\alpha\rangle_{L^2}
-\frac{1}{2}\langle(-N_{\alpha\beta}+\varepsilon\delta_{\alpha\beta})w_{,\alpha}w_{,\beta}\rangle_{L^2}.
\end{eqnarray}
From
$$-\text{ div }(z^*+Q)=h_{\alpha\beta\lambda\mu}w_{,\alpha\beta\lambda\mu}-P,$$
and
$$(z^*_\alpha)_{,\alpha}=[(-N_{\alpha\beta}+\varepsilon\delta_{\alpha\beta})(w_{,\beta})]_{,\alpha},$$
we obtain
\begin{eqnarray}-\text{ div }Q&=&h_{\alpha\beta\lambda\mu}w_{,\alpha\beta\lambda\mu}-P+\text{ div }(z^*)
\nonumber \\ &=&h_{\alpha\beta\lambda\mu}w_{,\alpha\beta\lambda\mu}-[(N_{\alpha\beta}-\varepsilon\delta_{\alpha\beta})(w_{,\beta})]_{,\alpha}-P. \end{eqnarray}

Let $v \in W^{2,2}_0(\Omega)$ be such that $Q=\nabla v$.

From the last equation $$ -\nabla^2 v=-\text{ div }(\nabla v)=-\text{ div } Q=h_{\alpha\beta\lambda\mu}w_{,\alpha\beta\lambda\mu}-[(N_{\alpha\beta}-\varepsilon\delta_{\alpha\beta})(w_{,\beta})]_{,\alpha}-P,$$
so that
$$v=-(\nabla^2)^{-1}(h_{\alpha\beta\lambda\mu}w_{,\alpha\beta\lambda\mu}-[(N_{\alpha\beta}-\varepsilon\delta_{\alpha\beta})(w_{,\beta})]_{,\alpha}-P)$$

Thus $$Q=\{[-(\nabla^2)^{-1}(h_{\alpha\beta\lambda\mu}w_{,\alpha\beta\lambda\mu}
-[(N_{\alpha\beta}-\varepsilon\delta_{\alpha\beta})(w_{,\beta})]_{,\alpha}-P]_{,\rho}\}$$

so that
\begin{eqnarray}
&&\frac{1}{2\varepsilon} \int_\Omega \delta_{\alpha\beta} Q_\alpha Q_\beta\;dx\nonumber \\ &=&\frac{1}{2 \varepsilon} \int_\Omega\left[(-\nabla^2)^{-1}\left(h_{\alpha\beta\lambda\mu} w_{,\alpha\beta\lambda\mu}
-[(N_{\alpha\beta}-\varepsilon \delta_{\alpha\beta})w_{,\beta}]_{,\alpha}-P\right)\right] \nonumber \\ && \times \left(h_{\alpha\beta\lambda\mu} w_{,\alpha\beta\lambda\mu}
-[(N_{\alpha\beta}-\varepsilon \delta_{\alpha\beta})w_{,\beta}]_{,\alpha}-P\right) \;dx.
\end{eqnarray}

Thus
\begin{eqnarray}\label{q12}\tilde{J}^*(Q,N)&=&-\frac{1}{2}\int_\Omega h_{\alpha\beta\lambda\mu} w_{,\alpha\beta}w_{,\lambda\mu}\;dx
\nonumber \\ &&-\frac{1}{2}\int_\Omega (N_{\alpha\beta}-\varepsilon \delta_{\alpha\beta})w_{,\alpha}w_{,\beta}\;dx \nonumber \\ &&-\frac{1}{2 \varepsilon} \int_\Omega\left[(-\nabla^2)^{-1}\left(h_{\alpha\beta\lambda\mu} w_{,\alpha\beta\lambda\mu}
-[(N_{\alpha\beta}-\varepsilon \delta_{\alpha\beta})w_{,\beta}]_{,\alpha}-P\right)\right] \nonumber \\ && \times \left(h_{\alpha\beta\lambda\mu} w_{,\alpha\beta\lambda\mu}
-[(N_{\alpha\beta}-\varepsilon \delta_{\alpha\beta})w_{,\beta}]_{,\alpha}-P\right) \;dx\nonumber \\ &&-\frac{1}{2}\int_\Omega \overline{H}_{\alpha\beta\lambda\mu} N_{\alpha\beta}N_{\lambda\mu}\;dx
\nonumber \\ &=& J_3(w,N).
\end{eqnarray}

Moreover, from $\delta J(u_0)=\mathbf{0}$ we have
$$h_{\alpha\beta\lambda\mu}(w_0)_{,\alpha\beta\lambda\mu}-[((N_0)_{\alpha\beta}-\varepsilon \delta_{\alpha\beta})(w_0)_{,\beta}]_{,\alpha}
-\varepsilon \delta_{\alpha\beta}(w_0)_{\alpha\beta}-P=0, \text{ in } \Omega,$$
so that
$$\hat{w}_0\equiv w_0=(\nabla^2)^{-1}(h_{\alpha\beta\lambda\mu}(w_0)_{,\alpha\beta\lambda\mu}
-[((N_0)_{\alpha\beta}-\varepsilon \delta_{\alpha\beta})(w_0)_{,\beta}]_{,\alpha}-P)/\varepsilon$$
and therefore
\begin{eqnarray}&&-h_{\alpha\beta\lambda\mu}(w_0)_{,\alpha\beta\lambda\mu}+[((N_0)_{\alpha\beta}-\varepsilon \delta_{\alpha\beta})(w_0)_{,\beta}]_{,\alpha}
\nonumber \\ &&+h_{\alpha\beta\lambda\mu}(\hat{w}_0)_{,\alpha\beta\lambda\mu}-[((N_0)_{\alpha\beta}-\varepsilon  \delta_{\alpha\beta})(\hat{w}_0)_{,\beta}]_{,\alpha}\nonumber \\ &=&0.\end{eqnarray}
Also,
$$\overline{H}_{\alpha\beta\lambda\mu}(N_0)_{\lambda\mu}=\frac{1}{2}\left(\frac{(u_0)_{\alpha,\beta}+(u_0)_{\beta,\alpha}}{2}\right)
+\frac{1}{2}(w_0)_{,\alpha}(w_0)_{,\beta},$$
so that
$$\delta\left\{J_3(w_0,N_0)+\left\langle\frac{1}{2}\left(\frac{(u_0)_{\alpha,\beta}+(u_0)_{\beta,\alpha}}{2}\right),(N_0)_{\alpha\beta}\right\rangle_{L^2}
-\langle (u_0)_{,\alpha},P_\alpha \rangle_{L^2}\right\}=\mathbf{0}.$$

From these last results and from the last theorem, we may obtain

\begin{eqnarray}
J_3(w_0,N_0)&=& J(w_0)
\nonumber \\ &=& J^*(z_0,Q_0,N_0) \nonumber \\ &=& \tilde{J}^*(Q_0,N_0).\end{eqnarray}

From this, also from the last theorem and from (\ref{q12}), we finally get

\begin{eqnarray}
J(u_0)&=&\inf_{u \in U} J(u) \nonumber \\ &=& \sup_{(Q,N) \in \hat{Y}_2^* \times A^*}\tilde{J}^*(Q,N)
\nonumber \\ &=& \tilde{J}^*(Q_0,N_0) \nonumber \\ &=& J^*(z_0^*,Q_0,N_0) \nonumber \\ &=& J_3(w_0,N_0) \nonumber \\ &=& \sup_{(w, N) \in U \times A^*} J_3(w,N).
\end{eqnarray}

The proof is complete.

\end{proof}
\section{ A multi-duality principle for non-convex optimization}

Our final result is a multi-duality principle, which is summarized by the following theorem.

\begin{thm} Considering the notation and statements of the plate model addressed in the last sections, assuming a not relabeled finite dimensional approximate model, in a finite elements or finite differences context, let $J_1: U \times Y_1^* \times Y_2^* \rightarrow \mathbb{R}$ be a functional
where
\begin{eqnarray}
J_1(u,Q,N)&=& \frac{1}{2}\int_\Omega h_{\alpha\beta\lambda\mu} w_{,\alpha\beta}w_{,\lambda\mu}\;dx -\langle P,w\rangle_{L^2}\nonumber \\
&&+\frac{1}{2}\int_\Omega (\overline{-N_{\alpha\beta}^K}) Q_\alpha Q_\beta \;dx-\langle w_{,\alpha},Q_\alpha \rangle_{L^2} \nonumber \\ &&
+\frac{K}{2}\int_\Omega w_{,\alpha}w_{,\alpha}\;dx -\frac{1}{2}\int_\Omega \overline{H}_{\alpha\beta\lambda\mu} N_{\alpha\beta}N_{\lambda\mu}\;dx
\nonumber \\ &&-\langle N_{\alpha\beta,\beta}+P_\alpha,u_\alpha \rangle_{L^2},
\end{eqnarray}
and where $$\{\overline{-N_{\alpha\beta}^K}\}=\{-N_{\alpha\beta}+K \delta_{\alpha\beta}\}^{-1}.$$
Define also,
$$C^*=\left\{ N \in Y_1^*\;:\; \{-N_{\alpha\beta}+K \delta_{\alpha\beta}\} > \left\{ \frac{K}{2} \delta_{\alpha\beta} \right\}\right\},$$

$$B^*=\{ N \in Y_1^*\;:\; N_{\alpha\beta,\beta}+P_\alpha=0, \text{ in } \Omega\},$$
$$D^+=\{N \in Y_1^*\;:\; \hat{J}^*_1(Q,N)>0,\; \forall Q \in Y_2^*\text{ such that } Q \neq \mathbf{0}\},$$
$$D^-=\{N \in Y_1^*\;:\; \hat{J}^*_2(Q,N)<0,\; \forall Q \in Y_2^*\text{ such that } Q \neq \mathbf{0}\},$$
where $$\hat{J}^*_1(Q,N)=-F^*_K(Q)+\frac{1}{2}\int_\Omega (\overline{-N_{\alpha\beta}^K}) Q_\alpha Q_\beta \;dx,$$
and where $$F_K^*(Q)=\sup_{ u \in U} \left\{\langle w_{,\alpha}, Q_\alpha \rangle_{L^2}
- \frac{1}{2}\int_\Omega h_{\alpha\beta\lambda\mu}w_{,\alpha \beta} w_{,\lambda\mu}\;dx-\frac{K}{2} \int_\Omega w_{,\alpha}w_{,\alpha}\;dx\right\}.$$
Moreover $$\hat{J}^*_2(Q,N)=-F^*_K(Q)+H_K^*(Q,N),$$ where
$$H_K^*(Q,N)=\sup_{ u \in U}\left\{ \langle w_{,\alpha},Q_{,\alpha} \rangle_{L^2}-\frac{1}{2} \int_\Omega (-N_{\alpha \beta}+K \delta_{\alpha\beta})w_{,\alpha}w_{,\beta}\;dx \right\}.$$

We also define,
$$A^*_+=B^*\cap C^* \cap D^+$$
$$A^*_-=B^* \cap C^* \cap D^-,$$
and
$$E^*=B^* \cap C^*.$$

Let $u_0 \in U$ be such that $\delta J(u_0)=\mathbf{0}$ and define
$$(N_0)_{\alpha\beta}=H_{\alpha\beta\lambda\mu} \gamma_{\lambda\mu}(u_0),$$
and $$(Q_0)_\alpha = (-(N_0)_{\alpha\beta}+K \delta_{\alpha\beta})(w_0)_{,\beta}.$$

Under such hypotheses,
\begin{enumerate}
\item if $\delta^2 J(u_0)> \mathbf{0}$ and $N_0 \in E^*$, defining
$$J_2(u,Q)=\sup_{ N \in E^*} J_1(u,Q,N),$$
and $$\tilde{J}^*(Q)=\inf_{u \in B_{r_1}(u_0)} J_2(u,Q),$$
where $r_1>0$ is such that $$\delta^2J(u)> \mathbf{0}$$ in $B_{r_1}(u_0)$,
we have  $$J(u_0)= \tilde{J}^*(Q_0),$$ $$\delta \tilde{J}^*(Q_0)=\mathbf{0}$$ and if $K>0$ is sufficiently big, $$\delta^2 \tilde{J}^*(Q_0)\geq  \mathbf{0}$$ and there  exist $r,\; r_2>0$ such that
\begin{eqnarray}
J(u_0)&=& \inf_{u \in B_r(u_0)}J(u) \nonumber \\ &=& \inf_{Q \in B_{r_2}(Q_0)} \tilde{J}^*(Q) \nonumber \\ &=& \tilde{J}^*(Q_0).
\end{eqnarray}
\item
If $N_0 \in A^*_+$, defining
$$J_3(u,Q)=\sup_{ N\in A^*_+} J_1(u,Q,N),$$
and $$\tilde{J}^*_3(Q)=\inf_{u \in U} J_3(u,Q),$$ then $$\delta \tilde{J}_3^*(Q_0)=\mathbf{0},$$
$$\delta^2 \tilde{J}^*_3(Q_0)\geq  \mathbf{0},$$ and
\begin{eqnarray} J(u_0)&=& \inf_{u \in U} J(u)\nonumber \\ &=& \inf_{ Q \in Y_2^*} \tilde{J}^*_3(Q) \nonumber \\ &=& \tilde{J}^*_3(Q_0).
\end{eqnarray}

\item
If $\delta^2 J(u_0)< \mathbf{0}$ so that $N_0 \in A^*_-$, defining
$$\hat{J}^*(Q,N)=-\hat{F}^*_K(Q)+H_K^*(Q,N)
-\frac{1}{2}\int_\Omega \overline{H}_{\alpha\beta\lambda\mu} N_{\alpha\beta}N_{\lambda\mu}\;dx,$$
 where $$\hat{F}_K^*(Q)=\sup_{ u \in U} \left\{\langle w_{,\alpha}, Q_\alpha \rangle_{L^2}
- \frac{1}{2}\int_\Omega h_{\alpha\beta\lambda\mu}w_{,\alpha \beta} w_{,\lambda\mu}\;dx-\frac{K}{2} \int_\Omega w_{,\alpha}w_{,\alpha}\;dx
+\langle w,P \rangle_{L^2}\right\}.$$
 we have that $$\hat{J}^*(Q_0,N_0)=J(u_0),$$ $$\delta \{\hat{J}^*(Q_0,N_0)-\langle (N_0)_{\alpha\beta,\beta}+P_\alpha, (u_0)_\alpha \rangle_{L^2}\}=\mathbf{0},$$ and
 there exist $r,\;r_1,\;r_2>0$ such that
\begin{eqnarray} J(u_0)&=& \sup_{u \in B_r(u_0)} J(u)\nonumber \\ &=& \sup_{ Q \in B_{r_1}(Q_0)}\left\{\sup_{ N \in B_{r_2}(N_0)\cap E^*}\hat{J}^*(Q,N)\right\}
\nonumber \\ &=& \hat{J}^*(Q_0,N_0).
\end{eqnarray}
\end{enumerate}
\end{thm}
\begin{proof}
From the assumption $N_0 \in E^*$ we have that
$$J_2(u_0,Q_0)=\sup_{ N \in E^*} J_1(u_0,Q_0,N)=J_1(u_0,Q_0,N_0),$$
where such a supremum is attained through the equation
$$\frac{\partial J_1(u_0,Q_0,N_0)}{\partial N}=\mathbf{0}.$$

Moreover, there exists $r>0,\; r_2>0$ such that $$J(u_0)=\inf_{ u \in B_r(u_0)}J(u),$$
and (we justify that the first infimum  $Q$ in this equation (\ref{a1}) is well defined in the next lines) \begin{eqnarray}\label{a1}J(u_0)&=&\inf_{u \in B_r(u_0)} \inf_{Q \in B_{r_2}(Q_0)} \sup_{N \in E^*} J_1(u,Q,N)
\nonumber \\ &=& \inf_{Q \in B_{r_2}(Q_0)}\inf_{u \in B_{r_1}(u_0)} \sup_{N \in E^*} J_1(u,Q,N) \nonumber \\ &=&
\inf_{ Q \in B_{r_2}(Q_0)} \tilde{J}^*(Q)\nonumber \\ &=& \tilde{J}^*(Q_0).\end{eqnarray}

Observe the concerning extremal in $Q$ is attained through the equation,
$$\frac{\partial J_1(u_0,Q_0,N_0)}{\partial Q}=\mathbf{0}.$$

Hence, from $$\frac{\partial J(u_0)}{\partial u}=\mathbf{0},$$ from the implicit function theorem and chain rule, we get
\begin{eqnarray}\mathbf{0}&=&\frac{\partial J(u_0)}{\partial u}\nonumber \\ &=& \frac{\partial J_1(u_0,Q_0,N_0)}{\partial u}
\nonumber \\ &&+\frac{\partial J_1(u_0,Q_0,N_0)}{\partial Q}\frac{\partial Q_0}{\partial u} \nonumber \\ &&
+ \frac{\partial J_1(u_0,Q_0,N_0)}{\partial N}\frac{\partial N_0}{\partial u} \nonumber \\ &=&\frac{\partial J_1(u_0,Q_0,N_0)}{\partial u}.
\end{eqnarray}
Therefore,
\begin{eqnarray}&&\frac{\partial \tilde{J}^*(Q_0)}{\partial Q}\nonumber \\ &=& \frac{\partial J_1(u_0,Q_0,N_0)}{\partial Q}
\nonumber \\ &&+\frac{\partial J_1(u_0,Q_0,N_0)}{\partial u}\frac{\partial u_0}{\partial Q} \nonumber \\ &&
+ \frac{\partial J_1(u_0,Q_0,N_0)}{\partial N}\frac{\partial N_0}{\partial Q} \nonumber \\ &=& \mathbf{0}.
\end{eqnarray}
From this we shall denote $$\delta \tilde{J}^*(Q_0)=\mathbf{0}.$$

Let us now show that the first infimum in $Q$ in (\ref{a1}) is well defined.

Recall again that, $$J_2(u_0,Q_0)=\sup_{ N \in E^*} J_1(u_0,Q_0,N),$$ where such a supremum is attained through the equation
$$\frac{\partial J_1(u_0,Q_0,N_0)}{\partial N}=\mathbf{0},$$
that is,
\begin{eqnarray}&&(Q_0)_\lambda \overline{(-N_0)_{\alpha \lambda}^K}\;\overline{(-N_0)_{\beta \mu}^K}(Q_0)_\mu
\nonumber \\ &&-\overline{H}_{\alpha\beta\lambda\mu} (N_0)_{\lambda \mu}+ \frac{(u_0)_{\alpha,\beta}+(u_0)_{\beta,\alpha}}{2} \nonumber
\\ &=& 0, \text{ in } \Omega.
\end{eqnarray}
Taking the variation of this last equation in $Q_\rho$ we have
\begin{eqnarray}&&-\overline{(-N_0)_{\alpha \rho}^K}\;\overline{(-N_0)_{\beta \mu}^K}(Q_0)_\mu \nonumber \\ &&
+(w_0)_\lambda \overline{(-N_0)_{\alpha \beta}^K}(w_0)_\mu \frac{\partial (N_0)_{\lambda \mu}}{\partial Q_\rho}
\nonumber \\ &&-\overline{H}_{\alpha\beta\lambda\mu} \frac{\partial (N_0)_{\lambda \mu}}{\partial Q_\rho} \nonumber
\\ &=& 0, \text{ in } \Omega,
\end{eqnarray}
that is
\begin{eqnarray}&&-\overline{(-N_0)_{\alpha \rho}^K}(w_0)_\beta \nonumber \\ &&
+(w_0)_\lambda \overline{(-N_0)_{\alpha \beta}^K}(w_0)_\mu \frac{\partial (N_0)_{\lambda \mu}}{\partial Q_\rho}
\nonumber \\ &&-\overline{H}_{\alpha\beta\lambda\mu} \frac{\partial (N_0)_{\lambda \mu}}{\partial Q_\rho} \nonumber
\\ &=& 0, \text{ in } \Omega,
\end{eqnarray}
so that
\begin{eqnarray}
&&\frac{\partial (N_0)_{\alpha\beta}}{\partial Q_\rho} \nonumber \\ &=&
\overline{\overline{H}_{\alpha\beta\lambda\mu}-(w_0)_\lambda \overline{(-N_0)_{\alpha\beta}^K} (w_0)_\mu}(\overline{(-N_0)_{\lambda \mu}^K}(w_0)_\rho).
\end{eqnarray}

Hence, if $K>0$ is sufficiently big, we obtain
\begin{eqnarray}
&&\left\{ \frac{\partial^2 J_2(u_0,Q_0)}{\partial Q_\alpha \partial Q_\beta} \right\}
\nonumber \\ &=& \left\{ \frac{\partial^2 J_1(u_0,Q_0,N_0)}{\partial Q_\alpha \partial Q_\beta} + \right.
\nonumber \\ && \left.+\frac{\partial^2 J_1(u_0,Q_0,N_0)}{\partial Q_\alpha \partial N_{\lambda \mu}} \frac{\partial (N_0)_{\lambda \mu}}{\partial Q_\beta}
\right\} \nonumber \\ &=& \{\overline{(-N_0)_{\alpha\beta}^K}\}+\overline{(-N_0)_{\alpha_\eta}^K} (w_0)_\rho \nonumber \\ &&
\times \left[\;\overline{\overline{H}_{\eta\rho\lambda\mu}\;-(w_0)_\lambda \;\overline{(-N_0)_{\eta\rho}^K}\; (w_0)_\mu}\;\right](\overline{(-N_0)_{\lambda \mu}^K}(w_0)_\beta)
\nonumber \\ &=& \{\overline{(-N_0)_{\alpha\beta}^K}\}+\mathcal{O}(1/K^2) \nonumber \\ &>& \mathbf{0}.
\end{eqnarray}

Therefore the first infimum in $Q$ in (\ref{a1}) is well defined.

Also, from (\ref{a1}) and the second order necessary condition for a local minimum, we obtain
$$\delta^2 \tilde{J}^*(Q_0)  \geq \mathbf{0}.$$

Assume now again $\delta J(u_0)=\mathbf{0}$ and $N_0 \in A^*_+.$

Recall that $$J_3(u,Q)=\sup_{ N \in A^*_+} J_1(u,Q,N).$$

Observe that if $N \in A^*_+$, by direct computation we may obtain
$$\delta^2_{uQ} J_1(u,Q,N)> \mathbf{0}.$$

Therefore $J_3(u,Q)$ is convex since is the supremum of a family of convex functionals.

Similarly as above we may obtain $$\delta J_3(u_0,Q_0)=\mathbf{0},$$ and
$$J_3(u_0,Q_0)=J(u_0)$$ so that
\begin{eqnarray}
J(u_0)&=&J_3(u_0,Q_0) \nonumber \\ &=& \inf_{(u,Q) \in U \times Y_2^*} J_3(u,Q)
\nonumber \\ &\leq& \inf_{Q \in Y_2^*} J_3(u,Q)\nonumber \\ &=& J(u),\; \forall u \in U.
\end{eqnarray}

Moreover, \begin{eqnarray} \tilde{J}^*_3(Q_0) &=& \inf_{ u \in U} J_3(u,Q_0) \nonumber \\ &=& J_3(u_0,Q_0)
 \nonumber \\ &\leq& J_3(u,Q), \forall u \in U,\; Q \in Y_2^*. \end{eqnarray}
 Hence,
 $$J(u_0)=\tilde{J}^*_3(Q_0)\leq \inf_{u \in U} J_3(u,Q)= \tilde{J}^*_3(Q),\; \forall Q \in Y_2^*.$$

 From these last results we may write,

 \begin{eqnarray}
 J(u_0)&=& \inf_{u \in U} J(u) \nonumber \\ &=&
 \inf_{Q \in Y_2^*} \tilde{J}^*_3(Q) \nonumber \\ &=& \tilde{J}^*_3(Q_0).
 \end{eqnarray}

 From this, similarly as above, we may obtain $$\delta^2 \tilde{J}^*_3(Q_0) \geq \mathbf{0}.$$

Finally, suppose now $\delta^2J(u_0)<\mathbf{0},$ so that $N_0 \in A^*_-.$

From this we obtain $$\frac{ \partial^2 \hat{J}^*(Q_0,N_0)}{\partial (Q_{\alpha,\alpha})^2}< \mathbf{0}$$ where, as previously indicated,
 $$\hat{J}^*(Q,N)=-\hat{F}^*_K(Q)+H_K^*(Q,N)
-\frac{1}{2}\int_\Omega \overline{H}_{\alpha\beta\lambda\mu} N_{\alpha\beta}N_{\lambda\mu}\;dx.$$

Here, $$H_K^*(Q,N)=\sup_{ u \in U}\left\{ \langle w_{,\alpha},Q_{\alpha} \rangle_{L^2}-\frac{1}{2} \int_\Omega (-N_{\alpha \beta}+K \delta_{\alpha\beta})w_{,\alpha}w_{,\beta}\;dx \right\}.$$

Denoting, \begin{eqnarray}
\hat{J}(u,N)&=& \frac{1}{2}\int_\Omega h_{\alpha\beta\lambda\mu} w_{,\alpha\beta} w_{,\lambda\mu}\;dx \nonumber \\
&&+ \frac{1}{2} \int_\Omega N_{\alpha \beta}w_{,\alpha}w_{,\beta}\;dx-\frac{1}{2}\int_\Omega \overline{H}_{\alpha\beta\lambda\mu}N_{\alpha\beta}N_{\lambda\mu}\;dx
\nonumber \\&& -\langle N_{\alpha\beta,\beta}+P_\alpha,u_\alpha \rangle_{L^2} ,\end{eqnarray}
also from $N_0 \in A^*_-$ and from $$\hat{J}^*(Q_0,N_0)=J(u_0),$$ $$\delta \{\hat{J}^*(Q_0,N_0)-\langle (N_0)_{\alpha\beta,\beta}+P_\alpha, (u_0)_\alpha \rangle_{L^2}\}=\mathbf{0},$$ (the proofs of such results are very similar to those of the corresponding cases developed above), there exist $r,r_1,r_2>0$ such that for $N \in A^*_- \cap B_{r_2}(N_0)$, we have

\begin{eqnarray} \sup_{u \in B_r(u_0)} \hat{J}(u,N)=\sup_{Q_\in B_{r_1}(Q_0)} \hat{J}^*(Q,N), \end{eqnarray}
and
\begin{eqnarray}
J(u_0)&=& \sup_{ u \in B_r(u_0)} J(u) \nonumber \\ &=& \sup_{ u \in B_r(u_0)}\left\{ \sup_{ N \in B_{r_2}(N_0) \cap E^*} \hat{J}(u,N) \right\}\nonumber \\ &=&
\sup_{ Q \in B_{r_1}(Q_0)}\left\{ \sup_{N \in B_{r_2}(N_0)\cap E^*} \hat{J}^*(Q,N)\right\} \nonumber \\ &=& \hat{J}^*(Q_0,N_0). \end{eqnarray}

The proof is complete.
\end{proof}

\section{Conclusion}
In this article we have developed a new primal dual variational formulation and a multi-duality principle applied to a non-linear model of plates.

About the primal dual formulation, we emphasize such a formulation is concave so that it is very interesting from a numerical analysis point of view.

Finally, the results here presented  may  be also developed in a similar fashion for a large class of problems,
including non-linear models in elasticity and other non-linear models of plates and shells.



\begin{thebibliography}{}
%
%
\bibitem{1} R.A. Adams and J.F. Fournier, Sobolev Spaces, 2nd edn. Elsevier, New York, 2003.

\bibitem{2900} W.R. Bielski, A. Galka, J.J. Telega, The Complementary Energy Principle and Duality for
Geometrically Nonlinear Elastic Shells. I. Simple case of moderate rotations around a tangent to the middle surface.
Bulletin of the Polish Academy of Sciences, Technical Sciences, Vol. 38, No. 7-9, 1988.
\bibitem{85} W.R. Bielski and J.J. Telega, A Contribution to Contact Problems for a Class of Solids and Structures,
Arch. Mech., 37, 4-5, pp. 303-320, Warszawa 1985.
\bibitem{120} F. Botelho, Functional Analysis and Applied Optimization in Banach Spaces,
 Springer Switzerland, 2014.


\bibitem{903} P.Ciarlet, {\it Mathematical Elasticity}, {Vol. I -- Three Dimensional Elasticity}, North Holland Elsevier (1988).

\bibitem{[3]} P.Ciarlet, {\it Mathematical Elasticity}, {Vol. II -- Theory of Plates}, North Holland Elsevier (1997).

\bibitem{[4]} P.Ciarlet, {\it Mathematical Elasticity}, {Vol. III -- Theory of Shells}, North Holland Elsevier (2000).

\bibitem{[6]} I. Ekeland, R. Temam, Convex Analysis and Variational Problems, North Holland, Amsterdam, 1976.
\bibitem{11} A.Galka and J.J.Telega {\it Duality and the complementary
energy principle for a class of geometrically non-linear structures.
Part I. Five parameter shell model; Part II. Anomalous dual
variational priciples for compressed elastic beams}, Arch. Mech. 47
(1995) 677-698, 699-724.

\bibitem{77} D.Y.Gao, {\it On the extreme variational principles for
non-linear elastic plates.} Quarterly of Applied Mathematics, {\bf XLVIII}, No. 2 (June 1990), 361-370.

\bibitem{9} D.Y.Gao, {\it Duality Principles in Nonconvex Systems,
Theory, Methods and Applications,} {Kluwer, Dordrecht},(2000).


\bibitem{29} R.T. Rockafellar, Convex Analysis, Princeton University Press, Princeton, 1970.
 \bibitem{10} J.J. Telega, {\it On the complementary energy principle in non-linear elasticity.
 Part I: Von Karman plates and three dimensional solids}, C.R. Acad.
 Sci. Paris, Serie II, 308, 1193-1198; Part II: Linear elastic solid
 and non-convex boundary condition. Minimax approach, ibid, pp.
 1313-1317 (1989)


\bibitem{12} J.F. Toland, {\it A duality principle for non-convex
optimisation and the calculus of variations}, Arch. Rath. Mech.
Anal., {\bf 71}, No. 1 (1979), 41-61.



\end{thebibliography}
\end{document}